\documentclass[11pt]{amsart} 
\usepackage{cases}
\usepackage{hyperref,mathtools}
\usepackage{amsmath,amssymb, amsthm, amsfonts}
\usepackage{graphicx}
\usepackage{caption}
\usepackage{float}
\usepackage{setspace}
\usepackage[dvipsnames]{xcolor}
\usepackage[margin=1in]{geometry}
\usepackage{epsf}
\usepackage{comment}
\usepackage{etoolbox}
\usepackage{makecell}
\usepackage{tikz}
\usetikzlibrary{shapes,arrows}
\usepackage{lmodern}
\usepackage{multirow}
\usepackage{soul}

\newtheorem{theorem}{Theorem}[section]
\newtheorem{corollary}[theorem]{Corollary}
\newtheorem{lemma}[theorem]{Lemma}
\newtheorem{prop}[theorem]{Proposition}

\theoremstyle{definition}

\newtheorem{defin}[theorem]{Definition}
\newtheorem{example}[theorem]{Example}
\newtheorem{rmk}[theorem]{Remark}

\newtheorem{obs}[theorem]{Observation}

\newcommand{\pf}[1]{\text{PF}_{#1}}
\newcommand{\upf}[1]{\text{UPF}_{#1}}
\newcommand{\rupf}[2]{\text{UPF}^{#2}_{#1}}
\newcommand{\fr}[1]{\text{FR}_{#1}}
\newcommand{\rfr}[2]{\text{FR}^{#2}_{#1}}
\newcommand{\rfb}[2]{\text{Fb}^{#2}_{#1}}
\newcommand{\fb}[1]{\text{Fb}_{#1}}

\makeatletter
\newcommand{\leqnomode}{\tagsleft@true}
\newcommand{\reqnomode}{\tagsleft@false}
\makeatother

\title{
Unit interval parking functions and the $r$-Fubini numbers
}

 \author[Bradt]{S. Alex Bradt}
 \address[S.~A.~Bradt]{School of Mathematical and Statistical Sciences, Arizona State University, Tempe, AZ 85281}
\email{\textcolor{blue}{\href{mailto:sbradt@asu.edu}{sbradt@asu.edu}}}

 \author[Elder]{Jennifer Elder}
 \address[J. Elder]{Department of Computer Science, Math \& Physics, Missouri Western State University, Saint Joseph, MO 64501 }
\email{\textcolor{blue}{\href{mailto:jelder8@missouriwestern.edu}{jelder8@missouriwestern.edu}}}

 \author[Harris]{Pamela E. Harris}
 \address[P.~E. Harris]{Department of Mathematical Sciences, University of Wisconsin-Milwaukee, Milwaukee, WI 53211}
 \email{\textcolor{blue}{\href{mailto:peharris@uwm.edu}{peharris@uwm.edu}}}
 \thanks{P.~E.~Harris was supported through a Karen Uhlenbeck EDGE Fellowship.}

 \thanks{This material is based upon work supported by the National Science Foundation under Grant No. DMS-1929284 while the authors were in residence at the Institute for Computational and Experimental Research in Mathematics in Providence, RI
 }

 \author[Rojas Kirby]{Gordon Rojas Kirby}
 \address[G.~Rojas Kirby]{Department of Mathematics and Statistics, San Diego State University, San Diego, CA 92182}
\email{\textcolor{blue}{\href{mailto:gkirby@sdsu.edu}{gkirby@sdsu.edu}}}

 \author[Reutercrona]{Eva Reutercrona}
 \address[E.~Reutercrona]{Pacific Lutheran University, Tacoma, WA 98447}
\email{\textcolor{blue}{\href{mailto:ereutercrona@plu.edu}{ereutercrona@plu.edu}}}

 \author[Wang]{Yuxuan (Susan) Wang}
 \address[Y.~Wang]{Mount Holyoke College, South Hadley, MA 01075}
\email{\textcolor{blue}{\href{mailto:}{wang264y@mtholyoke.edu}}}

 \author[Whidden]{Juliet Whidden}
 \address[J. Whidden]{Vassar College, Poughkeepsie, NY 12604}
\email{\textcolor{blue}{\href{mailto:jwhidden@vassar.edu}{jwhidden@vassar.edu}}}

\begin{document}

\begin{abstract}
We recall that unit interval parking functions of length $n$ are a subset of parking functions in which every car parks in its preference or in the spot after its preference, and Fubini rankings of length $n$ are rankings of $n$ competitors allowing for ties. 
We present an independent proof of a result of Hadaway, which establishes that unit interval parking functions and Fubini rankings are in bijection. 
We also prove that the cardinality of these sets are given by Fubini numbers. 
In addition, we give a complete characterization 
of unit interval parking functions by determining when a rearrangement of a unit interval parking function is again a unit interval parking function. 
This yields an identity for the Fubini numbers as a sum of multinomials over compositions. 
Moreover, we introduce a generalization of Fubini rankings, which we call the $r$-Fubini rankings of length $n+r$. 
We show that this set is in bijection with unit interval parking functions of length $n+r$ where the first $r$ cars have distinct preferences. 
We conclude by establishing that these sets are enumerated by the $r$-Fubini numbers. 
\end{abstract}


\maketitle


\section{Introduction}\label{sec:upf_intro}

Throughout we let $n\in\mathbb{N}=\{1,2,3,\ldots\}$ and $[n]=\{1,2,\ldots,n\}$.
In this work, we are concerned with studying and enumerating families of $n$-tuples in $[n]^n$ satisfying certain properties and providing bijections between them. Our main objects of study are unit interval parking functions and Fubini rankings. We begin by defining each as subsets of $[n]^n$ and recalling  some known results from the literature. 

Parking functions, which were introduced by Konheim and Weiss \cite{Konheim_Weiss} in the context of hashing problems, can be defined as $n$-tuples $\alpha=(a_1,\dots,a_n)\in [n]^n$ such that at least $i$ entries of $\alpha$ are at most $i$, for all $i\in[n]$. 
We denote the set of parking functions of length $n$ by $\pf{n}$, and it is well-known that $|\pf{n}|=(n+1)^{n-1}$, \cite{RiordanParking}.  

We consider a special subset of parking functions called unit interval parking functions, as defined by Hadaway in \cite{Hadaway_unit_interval}.
In this case, one considers a queue of $n$ cars entering a one-way street with $n$ parking spots labeled increasingly from $1$ to $n$. The preferences of the $n$ cars are provided in a preference list $\alpha=(a_1,a_2,\ldots,a_n)\in[n]^n$, where car $i$ prefers parking spot $a_i$, for all $i\in[n]$. 
Car $i$ drives to its preference $a_i$ and if that parking spot is available it parks there. 
Otherwise, it attempts to park in spot $a_i+1$. If that spot is available, then parking is successful. 
Otherwise, car $i$ fails to park.
If all cars are able to park using this unit interval parking rule, then we say that $\alpha$ is a \emph{unit interval parking function} of length $n$. 
For example, $(1,1,2)$ is a unit interval parking function while the rearrangement $(2,1,1)$ is not.
We let $\upf{n}$ denote the set of unit interval parking functions of length~$n$.

Now consider the set of possible rankings for how $n$ competitors can rank in a competition allowing ties. In this setting, no ranks can be skipped, whenever two competitors tie for a rank they ``cover'' that rank and the next, and similar whenever more than two competitors tie they cover. For example, $(2,3,5,3,1,5)$ is one such ranking, while $(1,1,4,4,5,6)$ is not.
These rankings were studied by Gross \cite{Gross_preferential_arrangements}, who showed that the number of such rankings is given by a Fubini number\footnote{The numbers  were named after Fubini by Comtet in \cite{Comtet_Fubini}}.
We recall that the Fubini numbers, also known as the ordered Bell numbers\cite[\href{https://oeis.org/A000670}{A000670}]{OEIS}, are defined by
\begin{align}
    \fb{n}=\sum _{{k=0}}^{n}\sum _{{j=0}}^{k}(-1)^{{k-j}}{\binom  {k}{j}}j^{n}.\label{fubini numbers}
\end{align}
Hadaway called these rankings ``Fubini rankings'' and established a bijection between them and unit interval parking functions \cite[Theorem~5.12]{Hadaway_unit_interval}. Her bijection depended on the content of a parking function, i.e.~the multiset describing the values appearing in the tuple.

Motivated by the results of Hadaway \cite{Hadaway_unit_interval}, in this work we give a direct bijection between unit interval parking functions and Fubini rankings, which highlights their ``block structure'' (Definition \ref{blocks}). 
The block structure of these combinatorial objects plays a key role in our answer of Hadaway's question of which rearrangements of a unit interval parking function is again a unit interval parking function. Moreover, this helped us uncover a new formula for the Fubini numbers. The block structure also allows us to fully generalize the result to a new family of combinatorial objects which we call $r$-Fubini rankings.

This article is organized as follows.
In Section~\ref{sec:preliminaries}, we study unit interval parking functions and we:
\begin{enumerate}
\item Give a complete characterization of unit interval parking functions by proving which rearrangements of unit interval parking functions are again unit interval parking functions (Theorem~\ref{thm:upf_rearrangement}).
    \item Establish a bijection between the set of unit interval parking functions of length $n$ and Fubini rankings of length $n$ (Theorem~\ref{thm:bijection upf and fr}). This implies immediately unit interval parking functions of length $n$ are enumerated by the Fubini numbers, Equation \eqref{fubini numbers}.
    As a consequence, we give a new formula for the Fubini numbers (Corollary~\ref{coro:new_fubini_formula}):
    \[\fb{n}=\displaystyle\sum_{k=1}^{n}\left(\sum_{\textbf{c}=(c_1,c_2,\ldots,c_k)\models n}\binom{n}{c_1,c_2,\ldots,c_k}\right), \]
    where $\textbf{c}=(c_1,c_2,\ldots,c_k)\models n$ denotes that $\textbf{c}$ is a composition of $n$ with $k$ parts.
    \end{enumerate}
In Section~\ref{sec:r fubini},
    we introduce the $r$-Fubini rankings of length $n+r$ (Definition~\ref{defin:r_Fubini_rankings}). Our main results  establish that they are enumerated by the $r$-Fubini numbers. We recall  
    that the 
$r$-Stirling numbers of second kind\footnote{The $2$-Stirling numbers of the second kind are \cite[\href{https://oeis.org/A143494}{A143494}]{OEIS}.} are defined by
\[\left\{\begin{matrix}n+r\\k+r\end{matrix}\right\}_r = \frac{1}{(k-r)!}\sum_{i=0}^{k-r} (-1)^{k+i+r}\binom{k-r}{i}(i+r)^{n-r},\]
and the $r$-Fubini numbers\footnote{Note that the $1$-Fubini numbers are the Fubini numbers defined in \eqref{fubini numbers}. } 
are a
generalization of the Fubini numbers defined in \cite{carlitz_1980,rFubini}
by 
\[\rfb{n}{r}=\sum_{k=0}^n (k+r)!\left\{\begin{matrix}n+r\\k+r\end{matrix}\right\}_r.\]
\begin{enumerate}\setcounter{enumi}{2}
    \item We establish that $r$-Fubini rankings are enumerated by the $r$-Fubini numbers (Theorem~\ref{thm:unit_interval_r_fub}) by giving a bijection between $r$-Fubini rankings and unit interval parking functions starting with $r$ distinct values (Theorem~\ref{thm:r-fubini_to_r-distinct_uipf}). 
\end{enumerate}
We end in Section~\ref{sec: open} detailing some directions for future research.

\section{Bijection between unit interval parking functions and Fubini rankings}\label{sec:preliminaries}

In this section, we establish the bijection between $\upf{n}$ and the set of Fubini rankings of length $n$, denoted by $\fr{n}$ (Theorem~\ref{thm:bijection upf and fr}). We remark that this result was originally established by Hadaway in \cite{Hadaway_unit_interval}. 
We begin by giving a formal definition of Fubini rankings.

\begin{defin}\label{defin:Fubini_rankings}
The $n$-tuple $\beta = (b_1,b_2, \ldots, b_n) \in [n]^n$ is a \emph{Fubini ranking} of length $n$ if the following holds:
For all $x \in [n]$, if $k>0$ entries of $\beta$ are equal to $x$, then the next largest value in $\beta$ is $x+k$.
Denote the set of Fubini rankings of length $n$ by $\fr{n}$.
\end{defin}

Note that $(1,2,3)$ is a Fubini ranking, while $(1,3,3)$ is not. 
The following result is due to Cayley~\cite{cayley_2009}. 

\begin{prop}
For $n\geq 1$, $|\fr{n}|=\fb{n}$. 
\end{prop}

We remark that any rearrangement of a Fubini ranking is again a Fubini ranking, as a Fubini ranking only depends on the ranks and not who places in which rank. Recall that this is also true for parking functions as any rearrangements of a parking function is also a parking function. 
However, this is not true for unit interval parking functions. For example, $(1,1,2)$ is a unit interval parking function while $(1,2,1)$ is not. We soon address the question of when a rearrangement of a unit interval parking function is again a unit interval parking function (Theorem \ref{thm:upf_rearrangement}). Before doing so, 
we consider Fubini rankings as parking preferences and establish that Fubini rankings are always parking functions. 

\begin{lemma}\label{lem:FRsArePFs}
For all $n\geq 1$, $\fr{n} \subseteq \pf{n}$.
\end{lemma}
\begin{proof}
    Suppose that $\beta=(b_1, \dots, b_n)$ is a Fubini ranking. Then by our previous remark we know that any rearrangement is also a Fubini ranking. Thus consider the weakly increasing rearrangement of $\beta$ which we call $\beta^\uparrow =(d_1,d_2,\ldots,d_n)$ which satisfies that $d_i\leq d_{i+1}$ for all $i\in[n-1]$.
    To show that $\beta\in \pf{n}$ it suffices to check that $d_i\leq i$ for all $i\in[n]$.
By Definition \ref{defin:Fubini_rankings},
there exists $j\in[n]$ values $d_1=d_2=\ldots=d_j=1$, thus, values $d_i\leq i$ for all $i\in[j]$.
Then, again by definition, the next entry is $d_{j+1}=j+1$, which satisfies $d_{j+1}\leq j+1$. Iterating this process, for $k> j+1$ if there are ties then $d_k< k$ or there is not tie and the value at $d_k$ is a new rank, hence $d_{k}=k\leq k$. Thus, $d_i\leq i$, for all $i\in[n]$,  and $\beta\in\pf{n}$.
\end{proof}

\begin{obs}\label{obs:FRBlockDoNotInteract}
Consider $\beta=(b_1,\dots, b_n)\in\fr{n}\subseteq [n]^n$ as a parking preference. Observe that cars at indices corresponding to a repeated value $j$ in $\beta$ do not interact with any car $c_i$ such that $b_i\neq j$. Thus, if $c_i$ has parking preference $j$, i.e. $b_i=j$, then $c_i$ parks in spot $j+k$, 
where $k$ denotes the number of occurrences of $j$ in $\beta$ up to car $c_i$.
\end{obs}

Next we provide an independent proof of the following result \cite[Theorem 5.12]{Hadaway_unit_interval}.

\begin{theorem}\label{thm:bijection upf and fr}
If $n\geq 1$, then the sets $\upf{n}$ and $\fr{n}$ are in bijection, and 
$|\upf{n}| = |\fr{n}| = \fb{n}$.
\end{theorem}

To make our bijection precise we begin by defining the displacement of a parking function: Given a parking function $\alpha=(a_1,a_2,\ldots,a_n)$, if car $i$ parks in spot $s_i$, then  $s_i-a_i$ is the \emph{displacement} of car $i$, 
and the sum $d(\alpha)=\sum_{i=1}^n (s_i-a_i)$ is called the \textit{total displacement} of $\alpha$.
We note that there are interesting results related to the study of parking functions and their displacement. This includes a bijection that takes the total displacement of parking functions to the number of inversions of labeled trees \cite{Kreweras}. 
In our work, displacement is utilized as a way to test whether a parking function is a unit interval parking function. Namely, we know that a parking functions is a unit interval parking function if and only if each car has displacement at most 1. In what follows, we say that whenever the displacement of a car is equal to zero, then the car is \textit{lucky}. Lucky cars and tree inversions were studied in \cite{GesselSeo,SeoShin}.

To prove Theorem~\ref{thm:bijection upf and fr} we 
define a bijection from Fubini rankings (which we know are parking functions by Lemma~\ref{lem:FRsArePFs}) to unit interval parking functions that modifies the entries of the Fubini ranking. 
This map will do this in such a way so as to ensure that each car still parks in the same space, but is displaced at most one, which makes it a unit interval parking function.

\begin{lemma}\label{lem: phi into U_n}
Given $\beta=(b_1,\ldots,b_n) \in \fr{n}$, define $\phi: \fr{n} \to \upf{n}$ by
$\phi(\beta)=(a_1,\dots,a_n) \in \upf{n}$, where

    $$a_i = \begin{cases}
        b_i & \text{if $|\{1\leq j\leq i:b_j=b_i\}|= 1$}\\
        b_i+k-2 & \text{if $|\{1\leq j\leq i:b_j=b_i\}|=k>1$.} 
    \end{cases}$$
The map $\phi$ is well defined.
\end{lemma}

\begin{proof} 
We show that $\phi(\beta)$ is a unit interval parking function whenever $\beta$ is a Fubini ranking. To do so we induct on $m$, the number of distinct values in $\beta$.

If $m=1$, then $\beta=(1,1\ldots,1)$ and $\phi(\beta)=(1,1,2,3,4,\ldots,n-1)$.
As a parking function, $\phi(\beta)$ parks the cars in order $1,2,\ldots, n$, and they all park within one spot of their preference. Thus $\phi(\beta)$ a unit interval parking function.

Suppose that if $\beta$ has $m$ distinct values, then $\phi(\beta)$ is a unit interval parking function. 
Now consider $\beta$ having $m+1$ distinct values. 

Let $I\subseteq[n]$ be the set of indices where the smallest $m$ distinct values of $\beta$ occur, and let $J=[n]\setminus I$ be the set of indices where the max value of $\beta$ occurs.
Note that by assumption, the entries in $\phi(\beta)$ indexed by $I$ park in the first $|I|$ parking spots, and they do so by parking at most one away from their preference. 

Now consider the entries in $\beta$ indexed by $J$, which contain the values $\max(\beta)$. 
Since $\beta$ is a Fubini ranking, we know that the next available rank is $|I|+1$.
Moreover note that $|I|+1=\max(\beta)$, as otherwise the rank $|I|+1$ would be unearned and $\beta$ would not have been a Fubini ranking. 
To finish building $\phi(\beta)$, 
we need only consider the values in index set $J$. 
By definition, those values in $\beta$ were all $\max(\beta)$ and in $\phi(\beta)$ they have now been replaced with the values $\max(\beta),\max(\beta),\max(\beta)+1,\ldots,\max(\beta)+|J|-2$, appearing in this relative order at the indices in $J$. 
We conclude by now noting that under $\phi(\beta)$ the cars indexed by $J$ park in spots $|I|+1,|I|+2,\ldots, n=\max(\beta),\max(\beta)+1,\ldots,n$. Thus, the cars indexed by $J$ parking under $\phi(\beta)$ are displaced at most one from their preference, as desired.
\end{proof}

In \cite[Lemma 5.6]{Hadaway_unit_interval} Hadaway established that if $\alpha \in \upf{n}$, then a value $a$ may appear at most twice. Below, we characterize when a rearrangement of a unit interval parking function is again a unit interval parking  function.

To make our approach precise, we begin with the following general definition on parking functions, which in Theorem~\ref{thm:upf_rearrangement} below we specialize to unit interval parking functions.
\begin{defin}\label{blocks}
Given $\alpha = (a_1, \ldots, a_n) \in \pf{n}$, let $\alpha^\uparrow = (a'_1, \ldots, a'_n)$ be the weakly increasing rearrangement of $\alpha$.
The partition of $\alpha^\uparrow$ as a concatenation denoted $\alpha'=\pi_1|\pi_2|\cdots|\pi_m$, where $\pi_j$ begins at (and includes) the $j$th entry $\alpha^\uparrow$ satisfying $a_i'=i$, is the \textit{block structure} of $\alpha$. Each $\pi_i$ is called a \textit{block} of $\alpha$.
\end{defin}
For example, if $\alpha=(8,1,5,5,1,2,4,7)\in\pf{8}$, then $\alpha^\uparrow=(1,1,2,4,5,5,7,8)$ and $\alpha'=112|4|55|7|8$.
\begin{obs}\label{obs:UPFblocksDoNotInteract}
Note that if $\alpha$ is a unit interval parking function with block structure $\pi_1|\pi_2|\cdots|\pi_m$, then as cars in each block park, they park without interacting with cars in other blocks. To make this precise we let $\ell(\pi_i)$ denote the length of $\pi_i$ for all  $i\in[m]$.
Then cars in $\pi_1$ park and occupy spots $1$ through $\ell(\pi_1)$.
By definition of $\pi_2$ car $c_{\ell(\pi_1)+1}$ has preference $\ell(\pi_1)+1$, and hence parks precisely in spot $\ell(\pi_1)+1$. 
As all other elements in $\pi_2$ are in weakly increasing order the remaining cars of $\pi_2$ park in spots $\ell(\pi_1)+2, \ell(\pi_1)+3,\ldots, \ell(\pi_1)+\ell(\pi_2)$. This shows that the cars in $\pi_1$ and $\pi_2$ park without affecting each other and this fact extends to all $\pi_j$ with $2<j\leq m$.

Lastly, we note that each block of $\alpha$ corresponds to a block of repeated values in a Fubini ranking, as described in Observation \ref{obs:FRBlockDoNotInteract}.
\end{obs}

\begin{theorem}\label{thm:upf_rearrangement}
Given $\alpha = (a_1, \ldots, a_n) \in \upf{n}$, let $\alpha^{\uparrow}$ be its weakly increasing rearrangement and $\alpha'=\pi_1|\pi_2|\dots|\pi_m$ be the block structure of $\alpha$ (as in Definition~\ref{blocks}).
\reqnomode
\begin{enumerate}
    \item  There are
\begin{align}
\binom{n}{|\pi_1|,\ldots, |\pi_m|}   \label{eq:rearrange count}
\end{align}
possible rearrangements $\sigma$ of $\alpha$ such that $\sigma$ is still a unit interval parking function.
\item A rearrangement $\sigma$ of $\alpha$ is in $\upf{n}$ if and only if the entries in $\sigma$ respect the relative order of the entries in each of the blocks $\pi_1,\pi_2,\ldots,\pi_m$.
\end{enumerate}
\end{theorem}

\begin{proof}
To establish Part (1) we prove the following claims by inducting on the number of blocks in $\alpha'$.
\begin{itemize}
    \item The weakly increasing rearrangement of an element $\alpha$ in $\upf{n}$ is also in $\upf{n}$.
    \item If $\pi_j$ is a weakly increasing block of $\alpha'$, then those elements appear in weakly increasing order in $\alpha$.
    \item Any two blocks $\pi_i$ and $\pi_j$ are independent of each other, i.e. as sets $\pi_i\cap\pi_j=\emptyset$ for all $i\neq j$.
    \item The formula in \eqref{eq:rearrange count} holds.
\end{itemize}

Suppose $m=1$. Then $\alpha^\uparrow=(a_1',\ldots,a_n')$ is such that $1\leq a_i'\leq a_{i+1}'$, $a_i'\leq i$ since it is a parking function, and $a_i'\neq i$ for $i>1$ by assumption. Thus, $a_1'=1$, which implies $a\leq a_2'<2$. That is $a_2'=1$ and car 2 parks in spot 2 and is displaced 1 spot from its preference. For car 3 to have displacement at most 1 we must have $2<a_3'$ since spots 1 and 2 are occupied so that $2\leq a_3'<3$, i.e. $a_3'=2$. Continuing in this fashion establishes $\alpha^\uparrow =(1,1,2,3\ldots,n-1)$. Next we establish that $\alpha=\alpha^{\uparrow}$ by supposing that $\alpha$ is out of weakly increasing order to obtain a contradiction.

Let $i$ be the smallest index such that $a_i>i-1$ and there exists some $k>i$ satisfying $a_i>a_j$. Let $j$ be the largest such $k$.
Parking under $\alpha$, cars $1,2,3,\ldots,i-1$ park in spots $1,2,\ldots,i-1$, respectively.
Car $i$ with preference $a_i>i-1$ parks in spot $a_i$ leaving spots $i,i+1,\ldots,a_i-1$ unoccupied after it parks.
Note that 
of the cars numbered $i+1, i+2, \ldots, a_i-a_j+i$, those with preferences in the set $\{i,i+1,\ldots,a_i-1\}$ park without displacement and those with preference smaller that $a_j$ get displaced one unit. Also, one of the cars with preference smaller than $a_j$ has parked in spot $a_j$.
Further, any car with preference larger than $a_j$ parks with a displacement at most one. This is implied because $\alpha$ consists of the numbers $1,1,2,3,\ldots,n-1$ so that all numbers not equal to one appear exactly once in $\alpha$.
This then ensures that spots $i,i+1,\ldots a_i-1$ are occupied by cars arriving and parking before car $j$. 
Thus, the displacement of car $j$ is at best $a_i-a_j+1>1+1\geq 2$, which contradicts that $\alpha\in UPF_n$. Therefore, the unique unit interval parking function consisting of a single block is $(1,1,2,\ldots, n-1)$, and it is in weakly increasing order.
Furthermore, there are $\binom{n}{|\pi_1|}=\binom{n}{n}=1$ possible rearrangements that produce a unit parking function.

Suppose that if $\alpha\in \upf{n}$ is such that $\alpha'$ has $m$ distinct blocks, then those blocks appear in $\alpha$ as weakly increasing strings, and there are $\binom{n}{|\pi_1|, \ldots , |\pi_m|}$, ways to reorder $\alpha$ so that the result is again in $\upf{n}$. 

Now consider $\alpha\in\upf{n}$ having $m+1$ distinct blocks in $\alpha'$. 
We may consider the block structure of $\alpha$ as the concatenation 
$\alpha'=\beta'\gamma'$, where the block structure of $\beta$ is given by 
$\beta' = \pi_1|\pi_2|\dots|\pi_m$, and the block structure of $\gamma$ is given by $\gamma' = \pi_{m+1}$. 
As in Observation \ref{obs:UPFblocksDoNotInteract} we may now consider both $\beta$ and $\gamma$ as unit parking functions on their own, also noting that $\max(\beta)+1=\min(\gamma)$, otherwise we would not have a parking function.

By induction, $\beta \in \upf{n-|\pi_{m+1}|}$, and there are $\binom{n-|\pi_{m+1}|}{|\pi_1|\ldots |\pi_m|}$ ways to rewrite it as a unit parking function, by induction hypothesis. 
Separate from this, $\gamma$ has the form of an element in $\upf{|\pi_{m+1}|}$, and there is one way to arrange these elements. 
Because $\gamma$ fills parking spots $n-|\pi_{m+1}|+1, n-|\pi_{m+1}|+2,\ldots, n$, it does not interfere with $\beta$ being a unit parking function, and thus $\beta$ and $\gamma$ can be parked ``without interactions.''
That is, cars can be given preferences from a valid rearrangement of $\beta$ at the same time as cars are given preferences from $\gamma$. Then if $\beta$ is a unit parking function produced by rearranging $\beta$, there are
\[
\binom{n}{n-|\pi_{m+1}|} \binom{|\pi_{m+1}|}{|\pi_{m+1}|}=\binom{n}{n-|\pi_{m+1}|, |\pi_{m+1}|}  
\]
ways that we could intermix the preferences from $\beta$ and $\gamma$. 
By induction, there are $ \binom{n-|\pi_{m+1}|}{|\pi_1|,\ldots ,|\pi_m|}$ ways to rearrange $\beta$ as a unit parking function, thus there are
\[
\binom{n-|\pi_{m+1}|}{|\pi_1|,\ldots, |\pi_m|} \binom{|\pi_{m+1}|}{|\pi_{m+1}|} = \binom{n}{|\pi_1|,\ldots, |\pi_m|, |\pi_{m+1}|}
\]
ways to rewrite $\alpha$ as a unit parking function, as desired.

To prove Part (2) first note that the above proof of Part (1) implies that if $\sigma$ respects the relative orders of the entries in the blocks, then $\sigma\in\upf{n} $. 
It suffices to prove the reverse direction, that given $\alpha$ a unit interval parking function, then it respects the relative order of the entries in each block.
To show this, note that the set of cars with preferences within a block of $\alpha$ park independently of each other. 
We know that every block of $\alpha$ must be a unit interval parking function in its own right, if $\alpha$ is a unit interval parking function. 
From the base case of Part (1), we established that a single block is a unit interval parking function if and only if its entries are in weakly increasing order. 
Thus, a unit interval parking function must have all of the values within its blocks appearing in weakly increasing order, as claimed.
\qedhere
\end{proof}

\begin{rmk}
Chaves Meyles, Harris, Jordaan, Rojas Kirby, Sehayek, and Spingarn give a very interesting and surprising connection between unit interval parking functions and the permutohedron \cite{unit_perm}. They also give an independent proof of Part (2) of Theorem~\ref{thm:upf_rearrangement} using a ``prime decomposition'' of parking functions. 
\end{rmk}

We are ready to introduce the inverse to the map in Lemma~\ref{lem: phi into U_n}.
\begin{lemma}\label{lem: psi into F_n}
Given $\alpha = (a_1, \ldots, a_n) \in \upf{n}$, let $\alpha^{\uparrow}$ be its weakly increasing rearrangement and $\alpha'=\pi_1|\pi_2|\dots|\pi_m$ be the block structure of $\alpha$ (as in Definition~\ref{blocks}).
Define $\psi: \upf{n} \mapsto  \fr{n}$ by
$\psi(\alpha) = (b_1, \dots b_n)$, where, for all $i\in[n]$,
$$b_i = \min\{a\in\pi_j: a_i \in \pi_j\}.$$

The map $\psi$ is well defined.
\end{lemma}
Before establishing Lemma~\ref{lem: psi into F_n} we provide an example to illustrate the map $\psi$.
\begin{example}

Let $n=10$ and consider $\alpha=(2,4,7,4,1,5,7,8,2,9)$.
We leave it to the reader to verify that $\alpha \in \upf{n}$.
The weakly increasing rearrangement $alpha^{\uparrow}$ and the block structure of $\alpha$, denoted $\alpha'$, are as follows:
\begin{align*}
\alpha^{\uparrow} &=(1,2,2,4,4,5,7,7,8,9)\\
\alpha'&=\underbrace{1}_{\pi_1} | \underbrace{22}_{\pi_2} | \underbrace{445}_{\pi_3} | \underbrace{7789}_{\pi_4}.
\end{align*}
For each $i\in[10]$, 
$b_i$ is the smallest element in the block containing the value $a_i$; that is, $\beta=(2,4,7,4,1,4,7,7,2,7)$.
One can verify that $\beta$ is a Fubini ranking.
\end{example}

\begin{proof}[Proof of Lemma~\ref{lem: psi into F_n}]
Let $\alpha=(a_1,a_2,\ldots,a_n)\in \upf{n}$ and let $\alpha^\uparrow=(a_1',a_2',\ldots,a_n')$ be the weakly increasing rearrangement of $\alpha$.

Let $I\subseteq[n]$ be the set of indices satisfying $a_i'=i$ and let $m = |I|$.
Then the block structure of  $\alpha$ can be partitioned as the concatenation $\alpha'=\pi_1|\pi_2|\dots|\pi_m$, where a new block $\pi_j$ begins at  (and includes) each $a'_i$ satisfying $a'_i=i$. 

Because $\alpha$ is a unit interval parking function, and thus $\alpha\uparrow$ is a parking function,
we note that $a'_h\leq h$ for all $h\in [n]$.
This means that if $a'_h=h$, and is the minimum element of a new block, this is the first occurrence of the number $h$ in $\alpha^\uparrow$.
Therefore, there are $h-1$ elements in $\alpha^\uparrow$ less than $h$, all appearing in the blocks before whichever block contains $h$.
Let $a'_h=h \in \pi_j$.
This integer, $h$, may then be partitioned as follows:
\[
h = 1+(h-1) =1+ \sum_{l=1}^{j-1} |\pi_l|.\]
Therefore, in general, we have that $$\min ( \pi_j)= 1+\sum_{l=1}^{j-1} |\pi_l|.$$

We also note that the elements $\min(\pi_j)$ and $\min(\pi_{j+1})$ then have the relationship $\min(\pi_j)+|\pi_j| = \min(\pi_{j+1})$.
Moreover, by definition of $b_i$, there are exactly $|\pi_j|$ copies of the value $\min(\pi_j)$ in $\psi(\alpha)$.
Also, note that by definition $\alpha^\uparrow$ begins with 1, and so $\min(\pi_1)=1$, and hence there exists a $b_j=1$ in $\psi(\alpha)$. These facts together imply that $\psi(\alpha)$ is a Fubini ranking. 
\end{proof}

The following result gives a  formula for the Fubini numbers, which we did not find in the literature.

\begin{corollary}\label{coro:new_fubini_formula}
If $n\geq 1$ and $\textbf{c}=(c_1,c_2,\ldots,c_k)\vDash n$ denotes a composition of $n$ with $k$ parts, then 
 \[\fb{n}=\displaystyle\sum_{k=1}^{n}\left(\sum_{\textbf{c}=(c_1,c_2,\ldots,c_k)\models n}\binom{n}{c_1,c_2,\ldots,c_k}\right). \]
\end{corollary}
\begin{proof}
    We count the number of Fubini rankings with $k$ blocks ($1 \leq k \leq n$), where blocks are defined as in Definition~\ref{blocks}.
    From Theorem~\ref{thm:upf_rearrangement}, we know that, for each composition $(c_1, c_2, \ldots, c_k)$ of $n$, the number of Fubini rankings with $c_i$ elements in the $i$th block is given by the multinomial $\binom{n}{c_1, c_2, \ldots, c_k}$.
\end{proof}

We are now ready to prove our main result.
\begin{proof}[Proof of Theorem~\ref{thm:bijection upf and fr}]
It suffices to show that $\phi$ and $\psi$ are inverses of each other. 

First, we prove that for every $\beta \in \fr{n}$, $\psi(\phi(\beta))=\beta$.
Suppose  $\beta$ has $m$ distinct values, $x_1<x_2<\cdots<x_m$. 

For each $i\in[m]$, let $I(x_i):=\{j\in[n]: b_j=x_i\}$ be the set of indices containing the value $x_i$ in $\beta$.
Since $\beta$ is a Fubini ranking, we have that
$1+\sum_{j=1}^{i-1}|I(x_j)|=x_{i}$ for all $i\in[m]$.
(In particular, $x_1 = 1$.)
For each $i\in[m]$, $\phi$ takes the values $x_i$ at index set $I(x_i)$ in $\beta$ and replaces them by the values $x_i,x_i,x_i+1,x_i+2,\ldots,x_i+|I(x_i)|-2$ at index set $I(x_i)$ (in that relative order) in $\phi(\beta)$.

To apply $\psi$ to $\phi(\beta)$, we first rearrange $\phi(\beta)$ into weakly increasing order, which is given by \[\phi(\beta)^\uparrow=(\underbrace{x_1,x_1,x_1+1,\ldots ,x_1+|I(x_1)|-2}_{|I(x_1)|},
\ldots, \underbrace{x_m,x_m,x_m+1,\ldots ,x_m+|I(x_m)|-2}_{|I(x_m)|}).\]
To partition the block structure of $\phi(\beta)$ we begin by noting that, for all $j\in[m]$, since $x_j=1+\sum_{k=1}^{j-1}|I(x_{k})|$ and the first $x_j$ appears in index $1+\sum_{k=1}^{j-1}|I(x_{k})|$, the first instance of $x_j$ begins a new block.
Thus
\[\phi(\beta)'=\pi_1|\pi_2|\cdots|\pi_m\]
where, for all $i\in[m]$, we have that
\begin{align}
\pi_j=\underbrace{x_j,x_j,x_j+1,\ldots ,x_j+|I(x_j)|-2}_{|I(x_j)|}.\label{eq:xj}
\end{align}
Now observe that the entries in $\psi(\phi(\beta))$ are defined by 
\[c_i=\min\{a\in\pi_j:a_i\in\pi_j\}.\]

\textbf{Claim:} $c_i=b_i$ for all $i\in[n]$.

To prove the claim we consider the values in the index set $I(x_i)$. 
In $\beta$, these are $x_i$'s. 
In $\phi(\beta)$, these are the values in $\pi_i$ which are $x_i,x_i,x_i+1,\ldots,x_i+|I(x_i)|-2$ (still occurring at the indices $I(x_i)$).
In $\psi(\phi(\beta))$ for any $i\in I(x_i)$, the value is $\min\{a\in\pi_i:a_i \in\pi_i\}=x_i$.
So for every $i\in I(x_i)$, $\psi(\phi(\beta))=x_i$ which is precisely the value of $\beta$ at every $i\in I(x_i)$.
Since this holds for all $i\in[m]$ and since $\cup_{i=1}^m I(x_i)=[n]$, we have established that $\psi(\phi(\beta))_i=b_i$ (showing the lists agree at every entry). Hence $\psi(\phi(\beta))=\beta$.

For the reverse composition, we begin with a unit interval parking function $\alpha=(a_1,a_2,\ldots,a_n)\in\upf{n}$ and let $\alpha^\uparrow=(a_1',a_2',\ldots,a_n')$ be the weakly increasing rearrangement of $\alpha$ and we partition the block structure of $\alpha$ as
\[\alpha'=\pi_1|\pi_2|\cdots|\pi_m\]
where a new block $\pi_j$ begins and includes each $a_i'$ satisfying $a_i'=i$.
For each $i\in[m]$, let $I(\pi_j)=\{i\in[n]:a_i\in\pi_j\}$.
Then use Lemma~\ref{lem: psi into F_n} to find $\psi(\alpha)=(b_1,b_2,\ldots,b_n)$ where $b_i=\min\{a\in\pi_j:a_i\in\pi_j\}$.
Note $\psi(\alpha)$ is a Fubini ranking.
Now consider $\phi(\psi(\alpha))=(c_1,c_2,\ldots,c_n)$.

\textbf{Claim:} $c_i=a_i$ for all $i\in[n]$.

To prove the claim we consider the values $c_i$ and $a_i$ where $i\in I(\pi_j)$ for an arbitrary $j\in[m]$.
Fix $j\in [m]$ and let $I(\pi_j)=\{k_1<k_2<\ldots<k_\ell\}$ where $\ell=|I(\pi_j)|$.
Let $\psi(\alpha)|_{I(\pi_j)}=\pi_{j,k_1}\pi_{j,k_2}\cdots \pi_{j,k_\ell}$ denote the substring of $\psi(\alpha)$ made up of the values at the indices in $I(\pi_j)$.
Note that by definition of $\psi$, $\pi_{j,k_b}=\min(\pi_j)$ for all $1\leq b\leq \ell$. 
Let $\phi(\psi(\alpha))|_{I(\pi_j)}=c_{k_1}c_{k_2}\cdots c_{k_\ell}$ denote the substring of $\phi(\psi(\alpha))$ made up of the values at the indices in $I(\pi_j)$.
By definition of $\phi$, 
\[c_{k_b}=\begin{cases}
\min(\pi_j)&\mbox{ if $b=1$}\\
\min(\pi_j)+b-2&\mbox{ if $b>1$.}
\end{cases}\]
By construction/definition
$a_{k_1}'a_{k_2}'\cdots a_{k_\ell}'=\pi_j$ and 
$a_{k_i}'=c_{k_i}$ for all $i\in[\ell]$. 
By Theorem~\ref{thm:upf_rearrangement}, these values appear in $\alpha$ in the exact same relative order. 
Thus, $a_{k_i}=c_{k_i}$ for all $i\in[\ell]$.
\end{proof}

\null
\section{Bijection between unit interval parking functions starting with \texorpdfstring{$r$}{r} distinct values and \texorpdfstring{$r$}{r}-Fubini rankings}\label{sec:r fubini}

We begin this section by defining $r$-Fubini numbers and show that they enumerate $r$-Fubini rankings.
We also establish that the $r$-Fubini rankings are in bijection with the set of unit interval parking functions, which begin with $r$ distinct values.

\begin{defin}\label{def:rFubini}
    The $r$-Fubini numbers\footnote{The $2$-Fubini numbers are defined in \cite[\href{https://oeis.org/A232472}{A232472}]{OEIS}.} are given by 
    \[\rfb{n}{r}=\sum_{k=0}^n (k+r)!\left\{\begin{matrix}n+r\\k+r\end{matrix}\right\}_r\]
    where $\left\{\begin{matrix}n+r\\k+r\end{matrix}\right\}_r$ denotes the $r$-Stirling numbers of the second kind\footnote{The $2$-Stirling numbers are defined in \cite[{\href{https://oeis.org/A008277}{A008277}}]{OEIS}.}, which enumerate set partitions of length $n+r$ into $k+r$ blocks, in which the first $r$ values appear in distinct blocks.
\end{defin}

 For ease of reference, we provide Table~\ref{tab:rfubini}, where we give the values of these numbers for $1\leq r\leq 8$, $1\leq m\leq 8$, and $m=n+r$. 
\begin{table}[!h]
\begin{center}
\begin{tabular}{ |c||c|c|c|c|c|c|c|c|}
\hline
$m=n+r\setminus r$&        1 & 2 & 3 & 4 & 5 & 6 & 7 & 8\\ \hline\hline
    
     1 & 1 & 0 & 0 & 0 & 0 & 0 & 0 & 0\\ \hline
    
     2 & 3 & 2 & 0 & 0 & 0 & 0 & 0 & 0 \\\hline
     3 & 13 & 10 & 6 & 0 & 0 & 0 & 0 & 0\\\hline
     4 & 75 & 62 & 42 & 24 & 0 & 0 & 0 & 0\\ \hline
     5 & 541 & 466 & 342 & 216 & 120& 0 & 0 & 0\\\hline
     6 & 4683 & 4142 & 3210 & 2184 & 1320 & 720 & 0 & 0\\\hline
     7 & 47293 & 42610 & 34326 & 24696 & 15960 & 9360 & 5040 & 0 \\\hline
     8 & 545835 & 498542 & 413322 & 310344 & 211560 & 131760 & 75600 & 40320\\\hline
\end{tabular}
\caption{The $r$-Fubini  numbers for $1\leq r\leq 8$, and $1\leq m\leq 8$, where $m=n+r$.}
\label{tab:rfubini}
\end{center}
\end{table}

\begin{defin}\label{defin:r_Fubini_rankings}
An \textbf{$\boldsymbol r$-Fubini ranking} of length $n+r$ is a Fubini ranking of length $n+r$ whose first $r$ values are distinct. More precisely, $\alpha=(a_1,a_2,\ldots,a_r,a_{r+1},\ldots,a_{n+r})\in\fb{n+r}$ is an $r$-Fubini ranking if $|\{a_1,a_2,\ldots,a_r\}|=r$.
 We let $\rfr{n+r}{r}$ denote the set of $r$-Fubini rankings of length $n+r$. 
\end{defin}

\begin{example}\label{ex:r_Fubini_rankings}
There are ten $2$-Fubini rankings of length three with the first two values being distinct. These are: $(1,3,1), (3,1,1), (1,2,2), (2,1,2), (1,2,3), (1,3,2), (2,1,3), (2,3,1), (3,1,2), (3,2,1)$. Note that this is precisely the value in Table~\ref{tab:rfubini} with $m=3$ and $r=2$.
\end{example}

Note that when $r=1$, the $1$-Fubini rankings allow the repetition of any of the values, and thus we recover the definition of Fubini rankings given by Hadaway \cite{Hadaway_unit_interval} and which we studied in Section~\ref{sec:preliminaries}. 
Hence, the $1$-Fubini rankings are enumerated by the Fubini numbers. 
Moreover, the $r$-Fubini rankings are nested, meaning that \[\mathfrak{S}_{m}=\rfr{m}{m}\subseteq \rfr{m}{m-1}\subseteq\cdots \subseteq \rfr{m}{2}\subseteq \rfr{m}{1}=\fr{m},\] where $\mathfrak{S}_m$ denotes the set of permutations on the set $[m]$.

In what follows we let $\rupf{m}{r}\subseteq\upf{m}$ denote the set of unit interval parking functions of length $m\coloneqq n+r$ in which the first $r$ values are distinct.
As with $r$-Fubini rankings, the sets $\rupf{m}{r}$ are also nested:
\[\mathfrak{S}_{m}=\rupf{m}{m}\subseteq \rupf{m}{m-1}\subseteq\cdots \subseteq \rupf{m}{2}\subseteq \rupf{m}{1}=\upf{m}.\]
Even though the sets share similar properties, such as this nesting, we now provide an example which illustrates that the sets $\rupf{n+r}{r}$ and $\rfr{n+r}{r}$ are not the same. 
\begin{example}
Note $(3,2,1,4,4,4)$ is in $\rfr{6}{4}$ but not in $\rupf{6}{4}$, since car $6$ would not park within a unit interval. Note
 $(3,2,1,4,4,5)$ is in $\rupf{6}{4}$ and not in $\rfr{6}{4}$, because because the tie at rank 4 would disallow rank 5 from appearing. 
\end{example}
We are now ready to state and prove our main result of this section.
\begin{theorem}\label{thm:r-fubini_to_r-distinct_uipf}
For $r\geq 1$ and $n\geq 0$, if $m=n+r$, then 
the sets $\rupf{m}{r} $ and  $\rfr{n}{r}$ are in bijection.
\end{theorem}
\begin{proof}
Given a $\upf{m}$ in which the first $r$ values are distinct, we use the bijection in Theorem~\ref{thm:bijection upf and fr} to find a unique $\fr{m}$.
We know that each of the first $r$ cars will park in its preferred spot, so each of those values will be starting a new $\pi$ block and will thus be mapped to itself in the Fubini ranking. 
So the Fubini ranking will start with the same $r$ distinct values as the unit interval parking function. 
Thus, $|\rupf{m}{r}|=|\rfr{n}{r}|$. 
\end{proof}

Next we enumerate the elements of $\rupf{n+r}{r}$.
\begin{theorem}\label{thm:unit_interval_r_fub}
If $r\geq 1$ and $n\geq 0$, then
     $|\rupf{n+r}{r}|=\rfb{n}{r}$.
\end{theorem}
\begin{proof}
    Let $\alpha=(a_1,a_2,\ldots,a_r,a_{r+1},\ldots,a_{n+r}) \in\rupf{n+r}{r}$ be weakly increasing. As $\rupf{n+r}{r}\subset\upf{n+r}$, by Theorem~\ref{thm:upf_rearrangement}
    we know the structure of $\alpha$, meaning that there exists  $r+k$, with $k\in[n]$, indices of $\alpha$ such that $a_i=i$. 
    These values always begin a block $\pi_i$. 
    By Theorem~\ref{thm:upf_rearrangement}, we know that we can construct all elements of $\rupf{n+r}{r}$ by permuting the blocks while keeping the relative order of the elements of each block. 
    The number of ways to do this, for a fixed $k\in[n]$, is given by 
    \[(k+r)!\left\{\begin{matrix}n+r\\k+r\end{matrix}\right\}_r\]
    where (as in Definition~\ref{def:rFubini}) the definition of the $r$-Stirling number ensures 
    that the first $k+r$ values appear in distinct blocks.
    The result follows by taking the sum over all possible $k\in[n]$. Therefore,
    \[|\rupf{n+r}{r}|=\sum_{k=0}^{n}(k+r)!\left\{\begin{matrix}n+r\\k+r\end{matrix}\right\}_r=\fb{n}^{r}.\qedhere\]
\end{proof}
Together Theorems~\ref{thm:r-fubini_to_r-distinct_uipf} and~\ref{thm:unit_interval_r_fub} imply the following.
\begin{corollary}
If $r\geq 1$ and $n\geq 0$, then
    $|\rfr{n+r}{r}|=\rfb{n}{r}$.
\end{corollary}

\section{Future Work}\label{sec: open}
To begin, we wonder if it is possible to construct a simple proof of Part 2 of Theorem \ref{thm:bijection upf and fr}, that does not rely on Part 1,  as that would eliminate some of the complexity in our argument. Moreover, 
there is much to be discovered about $r$-Fubini rankings and we now provide some directions for further study:
\begin{enumerate}
    \item We remark that the intersection between Fubini rankings and unit interval parking functions is non-trivial. In fact, computationally the sequence for $|\fr{n}\cap \upf{n}|$  is given by \cite[\href{https://oeis.org/A080599}{A080599}]{OEIS}  whose terms for $1\leq n\leq 7$ are:\[ 1, 3, 12, 66, 450, 3690, 35280.\] Stanley notes that this sequence gives the number of intervals in the weak (Bruhat) order of $\mathfrak{S}_n$ that are Boolean algebras. A new bijective proof has now appeared \cite{boolean}.
    \item One could extend the above to the case of $r$-Fubini rankings with $r>1$. Namely, enumerate the cardinalities of the set $\rfr{n}{r}\cap \rupf{n+r}{r}$, where we fix one of the parameters while varying the other. 
    \item The study of statistics of permutations is a well-researched mathematical area. Schumacher has results related to ascents, descents, and ties for parking functions \cite{Schumacher_Descents_PF} and  it would be of interest to study statistics for unit interval parking functions and Fubini rankings, as well as their generalizations $\rupf{n+r}{r}$ and $\rfr{n+r}{r}$. 
    We hope to follow up this paper with results in this direction. 
\end{enumerate}


\bibliographystyle{plain}
\bibliography{bibliography}

\end{document}